\documentclass{article}
\usepackage{amsmath,graphicx}
\usepackage{amsfonts}

\newtheorem{theorem}{Theorem}

\newtheorem{corollary}[theorem]{Corollary}

\newtheorem{example}[theorem]{Example}

\newtheorem{proposition}[theorem]{Proposition}

\newenvironment{proof}[1][Proof]{\noindent\textbf{#1.} }{\ \rule{0.5em}{0.5em}}

\begin{document}

\title{A new characterization of canal surfaces with parallel transport frame in Euclidean space $\mathbb{E}^{4}$}
\author{\.{I}lim K\.{I}\c{S}\.{I}, G\"{u}nay \"{O}ZT\"{U}RK \\
Department of Mathematics, Kocaeli University, Kocaeli, Turkey \\
ilim.ayvaz@kocaeli.edu.tr, ogunay@kocaeli.edu.tr\\
Kadri ARSLAN \\
Department of Mathematics, Uluda\u{g} University, Bursa, Turkey\\
arslan@uludag.edu.tr}
\maketitle

\begin{abstract}
In this study, we consider canal surfaces according to parallel
transport frame in Euclidean space $\mathbb{E}^{4}$. The curvature
properties of these surfaces are investigated with respect to
$k_{1}$, $k_{2}$ and $k_{3}$ which are principal curvature functions
according to parallel transport frame. We also give an example of
canal surfaces in $\mathbb{E}^{4}.$ Further, we point out that if
spine curve $\gamma $ is a straight line, then $M$ is a Weingarten
canal surface and also $M$ is a linear Weingarten tube surface.
Finally, the visualization of the projections of canal surfaces in
$\mathbb{E}^{3}$ are shown.
\end{abstract}

\noindent \textbf{Keywords:} Parallel transport frame, Gaussian curvature,
mean curvature\newline
\noindent \textbf{Classification:} [2010] Primary 53C40; Secondary 53C42

\section{Introduction}

Given a space curve $\gamma \left( u\right) $ called spine curve, a
canal surface associated to this curve is defined as a surface swept
by a family of spheres of varying radius $r(u)$. If $r(u)$ is
constant, the canal surface is called a tube or a pipe surface.

The canal surface can be thought out as a generalization of the
classical consept of an offset of a plane curve. In \cite{FN} and
\cite{FN1}, the analysis and algebraic features of offset curves are
discussed thoroughly. In \cite{C}, do Carmo gives some geometrical
properties of tube surfaces and by means of these surfaces proves
two very important theorems in differential geometry related to the
total curvature of space curves, named as Fenchel's theorem and the
Fary-Milnor theorem.

Apart from being used in pure mathematics, canal surfaces are widely used in
many areas especially in CAGD, e.g. construction of blending surfaces, i.e.
canal surface with a rational radius, shape reconstruction or robotic path
planning (see, \cite{FS}, \cite{SB}, \cite{WB}). Greater part of the studies
on canal surfaces within the CAGD context is related to the search of canal
surfaces with rational spine curve and rational radius function. Canal
surfaces are also useful in visualising long thin objects such as poles, 3D
fonts, brass instruments or internal organs of the body in solid/surface
modeling and CG/CAD.

Tori, Dupin cyclids in \cite{S} and tube surfaces in \cite{MPSY} are the
special types of the canal surfaces.

Given a surface $M$ in Euclidean space $\mathbb{E}^{4}$ and its two
principal curvatures $k_{1}$ and $k_{2},$ $M$ is a Weingarten surface under
the condition that there is a smooth relation $U(k_{1},k_{2})$ $=0.$ If $K$
and $H$ denote respectively the Gaussian curvature and the mean curvature of
$M$, $U(k_{1},k_{2})=0$ refers that $\Phi (K,H)=0$ which is equivalent to
the vanishing of the corresponding Jacobian determinant, i.e. $\left\vert
\frac{\partial \left( K,H\right) }{\partial \left( u,v\right) }\right\vert
=0.$ Also, if the surface satisfies a linear equation with respect to $K$
and $H$, that is, $aK+bH=c$; $a,b,c\in \mathbb{R}$, $((a,b,c)\neq (0,0,0))$,
then it is called as a linear Weingarten surface \cite{RY}.

Frenet-Serret frame gives way to the study of curves in classical
differential geometry in Euclidean space. However, the Frenet frame cannot
be constructed at the points in which curvature vanishes. Hence, an
alternative frame is needed. In \cite{B}, Bishop defined a new frame for a
curve and called it Bishop frame, which is well defined even if the curve's
second derivative in $3$-dimensional Euclidean space vanishes. In \cite{B,
KB} the advantages of the Bishop frame and the comparison of Bishop frame
with the Frenet frame in Euclidean $3$-space were given . Euclidean $4$%
-space $\mathbb{E}^{4}$ has the same problem as Euclidean 3-space. That is,
one of the $i-th$ ($1<i<4$) derivatives of the curve may be zero.

In \cite{GBGEY} using the similar idea authors considered such curves and
construct an alternative frame. They gave parallel transport frame of a
curve and introduced the relations between the frame and Frenet frame
of the curve in $4$-dimensional Euclidean space $\mathbb{E}^{4}$. They
generalized the notion which is well known in Euclidean $3$-space for $4$%
-dimensional Euclidean space $\mathbb{E}^{4}$.

In \cite{BABO} authors considered canal surfaces imbedded in an Euclidean
space of four dimensions. They investigated the curvature properties of
these surface with respect to the variation of the normal vectors and
curvature ellipse. They also gave some special examples of canal surfaces in
$\mathbb{E}^{4}$. Further, they gave necessary and sufficient condition for
canal surfaces in $\mathbb{E}^{4}$ to become superconformal.

In the present study, we consider canal surfaces imbedded in Euclidean $4$%
-space $\mathbb{E}^{4}$ with the spine curve $\gamma $ given with parallel
transport frame in $\mathbb{E}^{4}.$

This paper is organized as follows: Section $2$ gives some basic concepts of
the Frenet frame and parallel transport frame of a curve in $\mathbb{E}%
^{4}.$ Also this section provides some basic properties of canal surfaces in
$\mathbb{E}^{4}$ and the structure of their curvatures. Section $3$ tells
about the canal surfaces and some curvature conditions of these types of
surfaces in $\mathbb{E}^{4}$ according to parallel transport frame. In
section $4$, the visualization of canal surfaces are presented. All the
figures in this paper were generated via the Maple programme.

\section{Basic Concepts}

Let $\gamma =\gamma (s):I\rightarrow \mathbb{E}^{4}$ be a unit speed curve
in the Euclidean space $\mathbb{E}^{4}$, where $I$ is interval in $\mathbb{R}
$. Then the derivatives of the Frenet frame vectors of $\gamma $
(Frenet-Serret formula) are as follows;%
\begin{equation*}
\left[
\begin{array}{c}
T^{\prime } \\
N^{\prime } \\
B_{1}^{\prime } \\
B_{2}^{\prime }%
\end{array}%
\right] =\left[
\begin{array}{cccc}
0 & \kappa & 0 & 0 \\
-\kappa & 0 & \tau & 0 \\
0 & -\tau & 0 & \sigma \\
0 & 0 & -\sigma & 0%
\end{array}%
\right] \left[
\begin{array}{c}
T \\
N \\
B_{1} \\
B_{2}%
\end{array}%
\right] ,
\end{equation*}%
where $\left\{ T,N,B_{1},B_{2}\right\} $ is the Frenet frame of $\gamma $,
and $\kappa $, $\tau $ and $\sigma $ are principal curvature functions
according to Frenet frame of the curve $\gamma $, respectively.

In \cite{GBGEY}, authors used the tangent vector $T(s)$ and three relatively
parallel vector fields $M_{1}(s)$, $M_{2}(s)$, and $M_{3}(s)$ to construct
an alternative frame. They called this frame a parallel transport frame
along the curve $\gamma $. Then, they gave the following theorem for a
parallel transport frame.

\begin{theorem}
\cite{GBGEY} Let $\left\{ T,N,B_{1},B_{2}\right\} $ be the Frenet frame and $%
\left\{ T,M_{1},M_{2},M_{3}\right\} $ the parallel transport frame along a
unit speed curve $\gamma =\gamma (s):I\rightarrow \mathbb{E}^{4}$. The
relation between these frames may be expressed as%
\begin{eqnarray*}
T &=&T(s) \\
N &=&\cos \theta (s)\cos \psi (s)M_{1}+(-\cos \phi (s)\sin \psi (s)+\sin
\phi (s)\sin \theta (s)\cos \psi (s))M_{2} \\
&&+(\sin \phi (s)\sin \psi (s)+\cos \phi (s)\sin \theta (s)\cos \psi
(s))M_{3} \\
B_{1} &=&\cos \theta (s)\sin \psi (s)M_{1}+(\cos \phi (s)\cos \psi (s)+\sin
\phi (s)\sin \theta (s)\sin \psi (s))M_{2} \\
&&+(-\sin \phi (s)\cos \psi (s)+\cos \phi (s)\sin \theta (s)\sin \psi
(s))M_{3} \\
B_{2} &=&-\sin \theta (s)M_{1}+\sin \phi (s)\cos \theta (s)M_{2}+\cos \phi
(s)\cos \theta (s)M_{3},
\end{eqnarray*}%
where $\theta $, $\psi $ and $\phi $ are the Euler angles. Then the
alternative parallel frame equations are%
\begin{equation}
\left[
\begin{array}{c}
T~^{\prime } \\
M_{1}^{\prime } \\
M_{2}^{\prime } \\
M_{3}^{\prime }%
\end{array}%
\right] =\left[
\begin{array}{cccc}
0 & k_{1} & k_{2} & k_{3} \\
-k_{1} & 0 & 0 & 0 \\
-k_{2} & 0 & 0 & 0 \\
-k_{3} & 0 & 0 & 0%
\end{array}%
\right] \left[
\begin{array}{c}
T \\
M_{1} \\
M_{2} \\
M_{3}%
\end{array}%
\right] ,  \label{a01}
\end{equation}%
where $k_{1}$, $k_{2}$ and $k_{3}$ are principal curvature functions
according to parallel transport frame of the curve $\gamma $ and their
expressions are as follows:%
\begin{eqnarray*}
k_{1} &=&\kappa \cos \theta \cos \psi , \\
k_{2} &=&\kappa (-\cos \phi \sin \psi +\sin \phi \sin \theta \cos \psi ), \\
k_{3} &=&\kappa (\sin \phi \sin \psi +\cos \phi \sin \theta \cos \psi ),
\end{eqnarray*}%
where $\theta ^{\prime }=\frac{\sigma }{\sqrt{\kappa ^{2}+\tau ^{2}}}$, $%
\psi ^{\prime }=-\tau -\sigma \frac{\sqrt{\sigma ^{2}-\theta ^{^{\prime 2}}}%
}{\sqrt{\kappa ^{2}+\tau ^{2}}}$, $\phi ^{\prime }=-\frac{\sqrt{\sigma
^{2}-\theta ^{^{\prime 2}}}}{\cos \theta }$ and the following equalities%
\begin{eqnarray*}
\kappa  &=&\sqrt{k_{1}^{2}+k_{2}^{2}+k_{3}^{2}}, \\
\tau  &=&-\psi ^{\prime }+\phi ^{\prime }\sin \theta , \\
\sigma  &=&\frac{\phi ^{\prime }}{\sin \psi }, \\
\phi ^{\prime }\cos \theta +\theta ^{\prime }\cot \psi  &=&0
\end{eqnarray*}%
are hold.
\end{theorem}

Let $M$ be a regular surface in $\mathbb{E}^{4}$ given with the
parametrization $X(u,v)$ : $(u,v)\in D\subset \mathbb{E}^{2}$. The tangent
space of $M$ at an arbitrary point $p=X(u,v)$ is spanned by the vectors $%
X_{u}$ and $X_{v}$. The coefficients of the first fundamental form of $M$
are computed by
\begin{equation}
E=\langle X_{u},X_{u}\rangle ,F=\left\langle X_{u},X_{v}\right\rangle
,G=\left\langle X_{v},X_{v}\right\rangle ,  \label{A1}
\end{equation}%
where $\left\langle ,\right\rangle $ is the Euclidean inner product. We
assume that $W^{2}=EG-F^{2}\neq 0,$ i.e. the surface patch $X(u,v)$ is
regular.

For each $p$ in $M$, consider the decomposition $T_{p}\mathbb{E}%
^{4}=T_{p}M\oplus T_{p}^{\perp }M$ where $T_{p}^{\perp }M$ is the orthogonal
component of $T_{p}M$ in $\mathbb{E}^{4}.$ Let $\overset{\sim }{\nabla }$ be
the Riemannian connection of $\mathbb{E}^{4}$.

The induced Riemannian connection on $M$ for any given local vector fields $%
X_{1}$, $X_{2}$ tangent to $M$,  is defined by
\begin{equation}
\nabla _{X_{1}}X_{2}=(\widetilde{\nabla }_{X_{1}}X_{2})^{T},  \label{A2}
\end{equation}%
where $T$ represents the tangential component.

Let $\chi (M)$ and $\chi ^{\perp }(M)$ be the spaces of the smooth vector
fields tangent to $M$ and normal to $M$, respectively. The second
fundamental map is defined as follows:
\begin{eqnarray}
h &:&\chi (M)\times \chi (M)\rightarrow \chi ^{\perp }(M)  \notag \\
h(X_{i},X_{_{j}}) &=&\widetilde{\nabla }_{X_{_{i}}}X_{_{j}}-\nabla
_{X_{_{i}}}X_{_{j}}\text{ \ \ \ }1\leq i,j\leq 2.  \label{A3}
\end{eqnarray}%
This map is well-defined, symmetric and bilinear.

\begin{proposition}
\cite{Be} Let $M\subset E^{4}$ be a surface in $\mathbb{E}^{4}$ given with
the paramatrization $X\left( u,v\right) .$ If the coefficient of the first
fundamental form $F=0,$ the second fundamental forms of $M$ becomes%
\begin{eqnarray}
h(X_{u},X_{u}) &=&X_{uu}-\frac{1}{E}\left\langle \text{ }X_{uu},X_{u}\right%
\rangle X_{u}+\frac{1}{G}\left\langle X_{uv},X_{u}\right\rangle X_{v},
\notag \\
h(X_{u},X_{v}) &=&X_{uv}-\frac{1}{E}\left\langle X_{uv},X_{u}\right\rangle
X_{u}-\frac{1}{G}\left\langle X_{uv},X_{v}\right\rangle X_{v},  \label{F9} \\
h(X_{v},X_{v}) &=&X_{vv}+\frac{1}{E}\left\langle X_{uv},X_{v}\right\rangle
X_{u}-\frac{1}{G}\left\langle X_{vv},X_{v}\right\rangle X_{v}.  \notag
\end{eqnarray}
\end{proposition}

\begin{proposition}
\cite{Be} Let $M\subset E^{4}$ be a surface in $\mathbb{E}^{4}$ given with
the paramatrization $X\left( u,v\right) .$ Then for the basis $\left\{
X_{u},X_{v}\right\} $ of $T_{p}M,$ the Gaussian curvature\ and the mean
curvature vector of $M$ are defined as follows respectively,%
\begin{equation}
K=\frac{1}{W^{2}}\left( \left\langle
h(X_{u},X_{u}),h(X_{v},X_{v})\right\rangle -\left\langle
h(X_{u},X_{v}),h(X_{u},X_{v})\right\rangle \right)   \label{F10}
\end{equation}%
and%
\begin{equation}
\overrightarrow{H}=\frac{1}{2W^{2}}\left(
Eh(X_{v},X_{v})-2Fh(X_{u},X_{v})+Gh(X_{u},X_{u})\right) ,  \label{F11}
\end{equation}%
where $W^{2}=EG-F^{2}.$
\end{proposition}

\section{Canal Surfaces According to Parallel Transport Frame in $\mathbb{E}%
^{4}$}

Let $\gamma \left( u\right) =\left( \gamma _{1}\left( u\right) ,\gamma
_{2}\left( u\right) ,\gamma _{3}\left( u\right) ,\gamma _{4}\left( u\right)
\right) \subset \mathbb{E}^{4}$ be a curve parametrized by arclength.  The canal surface according to parallel transport frame has the following parametrization:

\begin{equation}
M:X\left( u,v\right) =\gamma \left( u\right) +r\left( u\right) \left(
M_{2}\left( u\right) \cos v+M_{3}\left( u\right) \sin v\right) ,  \label{a1}
\end{equation}
where $r(u)$ is a differentiable function and $\left\{ T,M_{1},M_{2},M_{3}\right\} $ is parallel transport
frame of the curve $\gamma$ in $\mathbb{E}^{4}$.
\begin{example}
\label{A}Consider the unit speed curve $\gamma \left( u\right) =\left( a\cos
cu,a\sin cu,b\cos du,b\sin du\right) $ in $\mathbb{E}^{4},$ where $%
a^{2}c^{2}+b^{2}d^{2}=1$. Then the canal surface associated to the spine
curve $\gamma $ in $\mathbb{E}^{4}$ has the following parametrization%
\begin{eqnarray*}
X\left( u,v\right)  &=&\left( a\cos cu+\frac{r\left( u\right) }{4\kappa }%
\left\{
\begin{array}{c}
\left( ac^{2}\cos cu+\sqrt{3}\kappa bd\sin cu-2\sqrt{3}bd^{2}\cos cu\right)
\cos v \\
+\left( -\sqrt{3}ac^{2}\cos cu-3\kappa bd\sin cu-2bd^{2}\cos cu\right) \sin v%
\end{array}%
\right\} \right. , \\
&&a\sin cu+\frac{r\left( u\right) }{4\kappa }\left\{
\begin{array}{c}
\left( ac^{2}\sin cu-\sqrt{3}\kappa bd\cos cu-2\sqrt{3}bd^{2}\sin cu\right)
\cos v \\
+\left( -\sqrt{3}ac^{2}\sin cu+3\kappa bd\cos cu-2bd^{2}\sin cu\right) \sin v%
\end{array}%
\right\} , \\
&&b\cos du+\frac{r\left( u\right) }{4\kappa }\left\{
\begin{array}{c}
\left( bd^{2}\cos du-\sqrt{3}\kappa ac\sin du+2\sqrt{3}ac^{2}\cos du\right)
\cos v \\
+\left( -\sqrt{3}bd^{2}\cos du+3\kappa ac\sin du+2ac^{2}\cos du\right) \sin v%
\end{array}%
\right\} , \\
&&\left. b\sin du+\frac{r\left( u\right) }{4\kappa }\left\{
\begin{array}{c}
\left( bd^{2}\sin du+\sqrt{3}\kappa ac\cos du+2\sqrt{3}ac^{2}\sin du\right)
\cos v \\
+\left( -\sqrt{3}bd^{2}\sin du-3\kappa ac\cos du+2ac^{2}\sin du\right) \sin v%
\end{array}%
\right\} \right) ,
\end{eqnarray*}%
where $\kappa =\sqrt{k_{1}^{2}+k_{2}^{2}+k_{3}^{2}},$ $0\leq u\leq 2\pi ,$ $%
a,b,c,d$ are real constants and $c,d>0.$
\end{example}

\begin{proposition}
Let $M$ be a canal surface in $\mathbb{E}^{4}$ according to parallel
transport frame given with the parametrization in (\ref{a1}). Then the
Gaussian curvature of $M$ at point $p$ is%
\begin{equation}
K=\frac{1}{r^{2}\left( f^{2}+r^{\prime ^{2}}\right) ^{2}}(f^{4}-f^{3}-fr%
\left( fr^{\prime \prime }-gr^{\prime }\right) -f_{v}^{2}r^{\prime ^{2}}).
\label{a2}
\end{equation}
\end{proposition}

\begin{proof}
Consider the parametrization (\ref{a1}). The partial derivatives of $%
X(u,v)$, which spans the tangent space of $M$, are expressed as%
\begin{eqnarray}
X_{u} &=&fT+r^{\prime }\cos vM_{2}+r^{\prime }\sin vM_{3},  \label{a3} \\
X_{v} &=&-r\sin vM_{2}+r\cos vM_{3},  \notag
\end{eqnarray}%
where $f=f\left( u,v\right) =1-k_{2}r\cos v-k_{3}r\sin v.$ Thus, the
coefficients of the first fundamental form become%
\begin{eqnarray}
E &=&\left\langle X_{u},X_{u}\right\rangle =f^{2}+r^{\prime ^{2}},  \notag \\
F &=&\left\langle X_{u},X_{v}\right\rangle =0,  \label{a4} \\
G &=&\left\langle X_{v},X_{v}\right\rangle =r^{2}.  \notag
\end{eqnarray}%
The second partial derivatives of $X(u,v)$ are expressed as follows:%
\begin{eqnarray}
X_{uu} &=&gT+fk_{1}M_{1}+\left( fk_{2}+r^{\prime \prime }\cos v\right)
M_{2}+\left( fk_{3}+r^{\prime \prime }\sin v\right) M_{3},  \notag \\
X_{uv} &=&f_{v}T-r^{\prime }\sin vM_{2}+r^{\prime }\cos vM_{3},  \label{a5}
\\
X_{vv} &=&-r\cos vM_{2}-r\sin vM_{3},  \notag
\end{eqnarray}%
where $g=g\left( u,v\right) =f_{u}-k_{2}r^{\prime }\cos v-k_{3}r^{\prime
}\sin v.$ Hence, from the equations (\ref{a3}) and (\ref{a5}), we get%
\begin{eqnarray}
\left\langle X_{uu},X_{u}\right\rangle  &=&ff_{u}+r^{\prime }r^{\prime
\prime },  \notag \\
\left\langle X_{uv},X_{u}\right\rangle  &=&ff_{v},  \label{a6} \\
\left\langle X_{uv},X_{v}\right\rangle  &=&rr^{\prime },  \notag \\
\left\langle X_{vv},X_{v}\right\rangle  &=&0.  \notag
\end{eqnarray}%
Further, by the use of equations (\ref{a3}), (\ref{a4}) and (\ref{a6}), the
second fundamental forms of $M$ become
\begin{eqnarray}
h\left( X_{u},X_{u}\right)  &=&\frac{1}{r\left( f^{2}+r^{\prime ^{2}}\right)
}\left( f^{2}r^{\prime }(f-1)-rr^{\prime }\left( fr^{\prime \prime
}-gr^{\prime }\right) \right) T  \notag \\
&&+fk_{1}M_{1}  \label{a7} \\
&&+\frac{f\cos v}{r\left( f^{2}+r^{\prime ^{2}}\right) }\left(
f^{2}-f^{3}+r\left( fr^{\prime \prime }-gr^{\prime }\right) \right) M_{2}
\notag \\
&&+\frac{f\sin v}{r\left( f^{2}+r^{\prime ^{2}}\right) }\left(
f^{2}-f^{3}+r\left( fr^{\prime \prime }-gr^{\prime }\right) \right) M_{3},
\notag \\
h\left( X_{u},X_{v}\right)  &=&\frac{f_{v}r^{\prime }}{f^{2}+r^{\prime ^{2}}}%
\left( r^{\prime }T-f\cos vM_{2}-f\sin vM_{3}\right) ,  \label{a7*} \\
h\left( X_{v},X_{v}\right)  &=&\frac{fr}{f^{2}+r^{\prime ^{2}}}\left(
r^{\prime }T-f\cos vM_{2}-f\sin vM_{3}\right) ,  \label{a7**}
\end{eqnarray}%
where $W^{2}=EG-F^{2}=r^{2}\left( f^{2}+r^{\prime ^{2}}\right) \neq 0.$ From
the equations (\ref{a7})-(\ref{a7**}) we get the result.
\end{proof}

As a conseguence of (\ref{a2}) we obtain the following result;

\begin{corollary}
Let $M$ be a tube surface with constant $r=r\left( u\right) .$ Then the
Gaussian curvature of $M$ becomes%
\begin{equation}
K=-\frac{k_{2}\cos v+k_{3}\sin v}{fr}=\frac{f-1}{fr^{2}}.  \label{a8}
\end{equation}
\end{corollary}

\begin{proposition}
Let $M$ be a canal surface in $\mathbb{E}^{4}$ according to parallel
transport frame given with the parametrization in (\ref{a1}). If $\gamma $
is a straight line, then the Gaussian curvature of $M$ at point $p$ is%
\begin{equation}
K=-\frac{r^{\prime \prime }}{r\left( 1+r^{\prime ^{2}}\right) ^{2}}.
\label{c1}
\end{equation}
\end{proposition}

\begin{proof}
Let $\gamma $ be a straight line, then the equations of parallel transport
frame of\ $\gamma $ become%
\begin{eqnarray*}
\gamma ^{\prime }\left( u\right) &=&T\left( u\right) , \\
T^{\prime }\left( u\right) &=&0, \\
M_{1}^{\prime }\left( u\right) &=&0, \\
M_{2}^{\prime }\left( u\right) &=&0, \\
M_{3}^{\prime }\left( u\right) &=&0.
\end{eqnarray*}%
Further, the tangent space of $M$ at an arbitrary point $p=X\left(
u,v\right) $ of $M$ is spanned by%
\begin{eqnarray}
X_{u} &=&T+r^{\prime }\cos vM_{2}+r^{\prime }\sin vM_{3},  \label{a11} \\
X_{v} &=&-r\sin vM_{2}+r\cos vM_{3}.  \notag
\end{eqnarray}%
Hence the coefficients of first fundamental form become%
\begin{eqnarray}
E &=&\left\langle X_{u},X_{u}\right\rangle =1+r^{\prime ^{2}},  \notag \\
F &=&\left\langle X_{u},X_{v}\right\rangle =0,  \label{a12} \\
G &=&\left\langle X_{v},X_{v}\right\rangle =r^{2}.  \notag
\end{eqnarray}%
The second partial derivatives of $X(u,v)$ are expressed as follows:%
\begin{eqnarray}
X_{uu} &=&r^{\prime \prime }\cos vM_{2}+r^{\prime \prime }\sin vM_{3},
\notag \\
X_{uv} &=&-r^{\prime }\sin vM_{2}+r^{\prime }\cos vM_{3},  \label{a13} \\
X_{vv} &=&-r\cos vM_{2}-r\sin vM_{3}.  \notag
\end{eqnarray}%
Thus from the equations (\ref{a11}) and (\ref{a13}), we get%
\begin{eqnarray*}
\left\langle X_{uu},X_{u}\right\rangle &=&r^{\prime }r^{\prime \prime }, \\
\left\langle X_{uv},X_{u}\right\rangle &=&0, \\
\left\langle X_{uv},X_{v}\right\rangle &=&rr^{\prime }, \\
\left\langle X_{vv},X_{v}\right\rangle &=&0.
\end{eqnarray*}%
Considering the equations (\ref{a11}), (\ref{a12}) and (\ref{a13}), we
obtain the second fundamental forms of $M$ as follows:
\begin{eqnarray}
h\left( X_{u},X_{u}\right) &=&-\frac{r^{\prime \prime }}{1+r^{\prime 2}}%
\left( r^{\prime }T-\cos vM_{2}-\sin vM_{3}\right) ,  \notag \\
h\left( X_{u},X_{v}\right) &=&0,  \label{a14} \\
h\left( X_{v},X_{v}\right) &=&\frac{r}{1+r^{\prime 2}}\left( r^{\prime
}T-\cos vM_{2}-\sin vM_{3}\right) ,  \notag
\end{eqnarray}%
where $W^{2}=EG-F^{2}=r^{2}\left( 1+r^{\prime ^{2}}\right) \neq 0.$ Hence
from the equations (\ref{a14}),\ we get the result.
\end{proof}

\begin{proposition}
Let $M$ be a canal surface according to parallel
transport frame given with the parametrization (\ref{a1}) in $\mathbb{E}^{4}$. When $\gamma $
is a straight line, the surface $M$ is a flat surface if and only if $r$ is
a linear function of the form $r(u)=au+b$ for some real constants $a$, $b$.
\end{proposition}

\begin{proposition}
Let $M$ be a canal surface in $\mathbb{E}^{4}$ according to parallel
transport frame given with the paramatrization in (\ref{a1}). Then the mean
curvature vector of $M$ at point $p$ is%
\begin{eqnarray}
\overrightarrow{H} &=&\frac{1}{2r\left( f^{2}+r^{\prime ^{2}}\right) ^{2}}%
\left\{ \left( fr^{\prime }\left( f^{2}+r^{\prime ^{2}}\right) -rr^{\prime
}(fr^{\prime \prime }-gr^{\prime })-f^{2}r^{\prime }\left( 1-f\right)
\right) T\right.  \notag \\
&&+frk_{1}\left( f^{2}+r^{\prime ^{2}}\right) M_{1}  \label{a9} \\
&&+\left( -f^{2}\cos v\left( f^{2}+r^{\prime ^{2}}\right) +f^{3}\cos v\left(
1-f\right) +fr(fr^{\prime \prime }-gr^{\prime })\cos v\right) M_{2}  \notag
\\
&&\left. +\left( -f^{2}\sin v\left( f^{2}+r^{\prime ^{2}}\right) +f^{3}\sin
v\left( 1-f\right) +fr(fr^{\prime \prime }-gr^{\prime })\sin v\right)
M_{3}\right\} .  \notag
\end{eqnarray}
\end{proposition}

\begin{proof}
Substuting the equations (\ref{a7})-(\ref{a7**}) into (\ref{F11}), we obtain
the vector given with (\ref{a9}).
\end{proof}

As a consequence of (\ref{a9}), we obtain the following results;

\begin{corollary}
Let $M$ be a canal surface in $\mathbb{E}^{4}$ according to parallel
transport frame given with the parametrization (\ref{a1}). Then the mean
curvature of $M$ at point $p$ is%
\begin{equation*}
H=\frac{1}{2r\left( f^{2}+r^{\prime ^{2}}\right) ^{\frac{3}{2}}}\left(
\begin{array}{c}
f^{2}\left( f^{2}+r^{\prime ^{2}}\right) ^{2}-2frr^{\prime ^{2}}(fr^{\prime
\prime }-gr^{\prime })-2f^{3}\left( f^{2}+r^{\prime ^{2}}\right) \left(
1-f\right) \\
+(fr^{\prime \prime }-gr^{\prime })^{2}r^{2}+2f^{2}r(fr^{\prime \prime
}-gr^{\prime })+f^{4}\left( 1-f\right) ^{2} \\
+f^{2}r^{2}k_{1}^{2}\left( f^{2}+r^{\prime ^{2}}\right) -4f^{3}r(fr^{\prime
\prime }-gr^{\prime })%
\end{array}%
\right) ^{\frac{1}{2}}.
\end{equation*}
\end{corollary}

\begin{corollary}
Let $M$ be a tube surface with constant $r=r\left( u\right) .$ Then the mean
curvature vector of $M$ becomes%
\begin{equation}
\overrightarrow{H}=\frac{1}{2fr}\left(
\begin{array}{c}
rk_{1}M_{1} \\
+\left( -2f\cos v+\cos v\right) M_{2} \\
+\left( -2f\sin v+\sin v\right) M_{3}%
\end{array}%
\right) .  \label{a10}
\end{equation}
\end{corollary}

\begin{corollary}
Let $M$ be a tube surface with constant $r=r\left( u\right) .$ Then the mean
curvature of $M$ at point $p$ is%
\begin{equation*}
H=\frac{1}{2fr}\left( 4f^{2}-4f+r^{2}k_{1}^{2}+1\right) ^{\frac{1}{2}}.
\end{equation*}
\end{corollary}

\begin{proposition}
Let $M$ be a canal surface in $\mathbb{E}^{4}$ according to parallel
transport frame given with the parametrization in (\ref{a1}). If $\gamma $
is a straight line then the mean curvature vector of $M$ at point $p$ is
\begin{equation}
\overrightarrow{H}=\frac{1+r^{\prime ^{2}}-rr^{\prime \prime }}{2r\left(
1+r^{\prime ^{2}}\right) ^{2}}\left( r^{\prime }T-\cos vM_{2}-\sin
vM_{3}\right) .  \label{a15}
\end{equation}
\end{proposition}

\begin{proof}
Considering the equations (\ref{a12}) and (\ref{a14}), we obtain the
solution.
\end{proof}

\begin{corollary}
Let $M$ be a canal surface in $\mathbb{E}^{4}$ according to parallel
transport frame given with the parametrization in (\ref{a1}). If $\gamma $
is a straight line then the mean curvature of $M$ at point $p$ is%
\begin{equation}
H=\frac{r^{\prime ^{2}}-r^{\prime \prime }r+1}{2r\left( 1+r^{\prime
^{2}}\right) ^{\frac{3}{2}}}.  \label{d1}
\end{equation}
\end{corollary}

\begin{proposition}
Let $M$ be a canal surface in $\mathbb{E}^{4}$ according to parallel
transport frame given with the parametrization in (\ref{a1}). If $\gamma $
is a straight line, the surface $M$ is minimal if and only if
\begin{equation*}
2r+2\sqrt{r^{2}-c_{1}^{2}}=e^{\frac{u}{c_{1}}+c_{2}}.
\end{equation*}
\end{proposition}

\begin{proof}
Let $M$ is minimal. Then from the equation (\ref{d1}), $r^{\prime
^{2}}-r^{\prime \prime }r+1=0.$ If we take $r^{\prime }=p\left( u\right) $,
the last equation becomes%
\begin{equation}
\frac{dr}{r}=\frac{pdp}{p^{2}+1}.  \label{a16}
\end{equation}%
The solution of the equation (\ref{a16}) is as follows:%
\begin{equation*}
r^{2}=\left( p^{2}+1\right) c_{1}^{2}.
\end{equation*}%
Again taking $p\left( u\right) =r^{\prime }$, we obtain the following
ordinary differential equation:%
\begin{equation*}
\frac{dr}{\sqrt{r^{2}-c_{1}^{2}}}=\frac{du}{c_{1}}.
\end{equation*}
Integrating both sides of the last equation, we get the solution.
\end{proof}

As a consequence of (\ref{a15}), we obtain the following result;

\begin{proposition}
Let $M$ be a tube surface in $\mathbb{E}^{4}$ according to parallel
transport frame given with the parametrization in (\ref{a1}). If $\gamma $
is a straight line, the mean curvature vector of $M$ at point $p$ is
\end{proposition}

\begin{equation*}
\overrightarrow{H}=\frac{1}{2r}\left( -\cos vM_{2}-\sin vM_{3}\right) .
\end{equation*}

\begin{corollary}
Let $M$ be a tube surface in $\mathbb{E}^{4}$ according to parallel
transport frame given with the parametrization in (\ref{a1}). If $\gamma $
is a straight line, $M$ has constant mean curvature of the form%
\begin{equation*}
H=\frac{1}{2r}=\text{constant}.
\end{equation*}
\end{corollary}

\begin{proposition}
Let $M$ be a canal surface in $\mathbb{E}^{4}$ according to parallel
transport frame given with the parametrizetion in (\ref{a1}). If $\gamma $
is a straight line, then $M$ is a Weingarten surface.
\end{proposition}

\begin{proof}
Considering the equations (\ref{c1}) and (\ref{d1}), we see that $K$ and $H$
are functions of the variable $u.$ So%
\begin{equation*}
K_{v}=0=H_{v},
\end{equation*}%
which means $K_{u}H_{v}-K_{v}H_{u}=0.$
\end{proof}

\begin{proposition}
Let $M$ be a tube surface in $\mathbb{E}^{4}$ according to parallel
transport frame given with the parametrization in (\ref{a1}). If $\gamma $
is a straight line, then $M$ is a \ linear Weingarten surface.
\end{proposition}

\begin{proof}
Let $M$ be a tube surface in $\mathbb{E}^{4}$ according to parallel
transport frame given with the parametrization in (\ref{a1}) and $\gamma $
is a straight line, then we know that $K=0$ and $H=\frac{1}{2r}.$ Then for $%
a,b,c\in \mathbb{R}$, we get
\begin{equation*}
a.0+b.\frac{1}{2r}=c,
\end{equation*}%
which has the solution $\left( a,b,c\right) =\left( t,2rk,k\right) ;t,k\in
\mathbb{R}$.
\end{proof}

\section{Visualization}

Canal surfaces are very popular in geometric modelling. In this section, we
visualize the surfaces given with the patch
\begin{equation*}
X(u,v)=(x(u,v),y(u,v),z(u,v),w(u,v))
\end{equation*}%
in $\mathbb{E}^{4}$ by the use of Maple Program. We plot the graph of the
given surface by using maple plotting command%
\begin{equation}
plot3d([x,y,z+w],u=a..b,v=c..d);  \label{C0}
\end{equation}

After than, we construct some 3D geometric shape models by using the canal
surfaces defined in Example \ref{A} for the following values;%
\begin{eqnarray*}
a)\text{ \ }r(u) &=&2u+6, \\
b)\text{ \ }r(u) &=&u^{2}, \\
c)\text{ \ }r(u) &=&cosu^{2}.
\end{eqnarray*}%
We plot the graph of the projection of these surfaces in $\mathbb{E}^{3}$ by
the use of plotting command (\ref{C0}) (see, Figure 1).

\begin{figure}[thbp]
\centering
\includegraphics[height=8cm,width =8cm,angle=0]{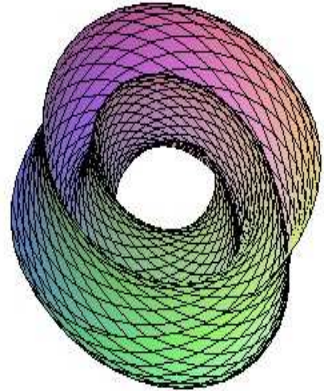}
\caption{Canal surfaces with $r(u)=2u+6$}
\end{figure}

\begin{figure}[!htb]
\begin{center}
\includegraphics[width=8cm]{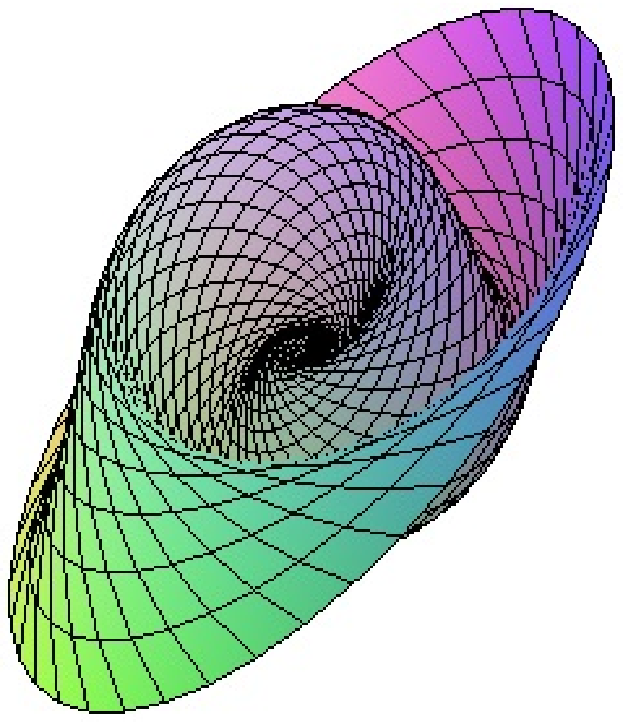}
\end{center}
\caption{Canal surfaces with $r(u)=u^{2}$}
\end{figure}

\begin{figure}[!htb]
\begin{center}
\includegraphics[width=8cm]{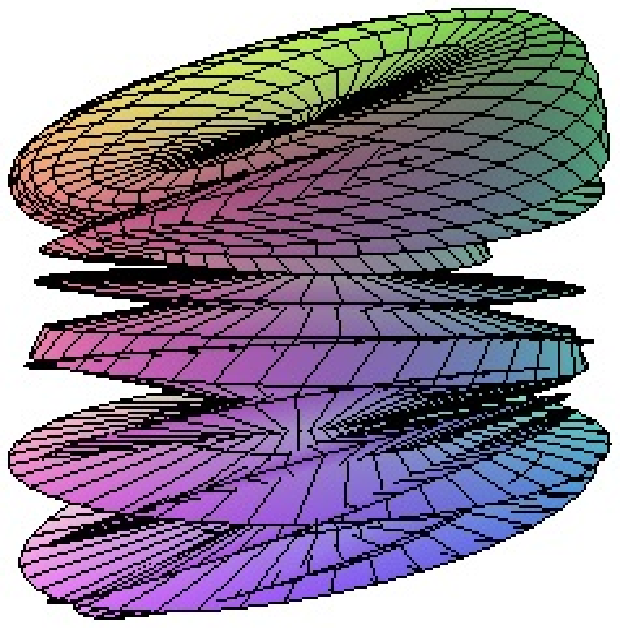}
\end{center}
\caption{Canal surfaces with $r(u)=cosu^{2}$}
\end{figure}


\begin{thebibliography}{99}
\bibitem{B} Bishop LR. There is more than one way to frame a curve. Amer
Math Monthly 1975; 82: 246-251.

\bibitem{Be} Bulca B. A characterization of surfaces in $\mathbb{E}^{4}$.
PhD, Uluda\u{g} University, Bursa, TURKEY, 2012.

\bibitem{BABO} Bulca B, Arslan K, Bayram B, \"{O}zt\"{u}rk G. Canal surfaces
in 4-dimensional Euclidean space, Preprint.

\bibitem{C} Do Carmo PM. Differential Geometry of Curves and Surfaces.
Englewood Cliffs, NJ, USA: Prentice-Hall, 1976.

\bibitem{FS} Farouki RT, Sverrissor R. Approximation of rolling-ball blends
for free-form parametric surfaces. Computer-Aided Design 1996, 28: 871-878.

\bibitem{FN} Farouki RT, Neff CA. Analytic properties of plane offset
curves. Computer-Aided Geometric Design 1990, 7: 83-99.

\bibitem{FN1} Farouki RT Neff CA. Algebraic properties of plane offset
curves, Computer-Aided Geometric Design 1990, 7: 101-127.

\bibitem{GBGEY} G\"{o}k\c{c}elik F, Bozkurt Z, G\"{o}k I, Ekmekci FN,
Yayl\i\ Y. Parallel transport frame in 4-dimensional Euclidean space $%
\mathbb{E}^{4}$. Caspian J of Math Sci 2014, 3: 91-103.

\bibitem{KB} Karacan MK, B\"{u}kc\"{u} B. On natural curvatures of Bishop
frame. Journal of Vectorial Relativity 2010, 5: 34-41.

\bibitem{MPSY} Maekawa T, Patrikalakis NM, Sakkalis T, Yu G. Analysis and
applications of pipe Surfaces. Computer-Aided Geometric Design 1998, 15:
437-58.

\bibitem{RY} Ro SJ, Yoon DW. Tubes of Weingarten types in a Euclidean
3-Space. Journal of the Chungcheong Mathematical Society 2009, 22: 359-366.

\bibitem{SB} Shani U, Ballard DH. Splines as embeddings for generalized
cylinders, Computer Vision Graphics and Image Processing 1984, 27: 129-156.

\bibitem{S} Shene CK. Blending two cones with Dupin cyclids. Computer-Aided
Geometric Design 1998, 15: 643-73.

\bibitem{WB} Wang L, Ming CL, Blackmore D. Generating sweep solids for NC
verification using the SEDE method. Proceedings of the Fourth Symposium on
Solid Modeling and Applications; 14-16 May 1995; Atlanta. Georgian: pp.
364-375.
\end{thebibliography}
\end{document}